\documentclass{amsart}

\input xy
\xyoption{all}

\usepackage{geometry}
\usepackage{amsmath}
\usepackage{amssymb}
\usepackage{cite}
\usepackage{amsthm}  % see geometry.pdf on how to lay out the page. There's lots.
\usepackage{tikz}
\usetikzlibrary{er,positioning}

\newtheorem{same}{This should never appear}[section]
\newtheorem{defin}[same]{Definition}

\newtheorem{remark}[same]{Remark}
\newtheorem{theorem}[same]{Theorem}
\newtheorem{example}[same]{Example}
\newtheorem{lemma}[same]{Lemma}
\newtheorem{fact}[same]{Fact}
\newtheorem{question}[same]{Question}
\newtheorem{cor}[same]{Corollary}

\newtheorem{hypothesis}[same]{Hypothesis}

\newtheorem*{mainm}{Main Theorem}
\newtheorem*{sec-1}{Shelah's categoricity conjecture for AECs}

\newtheorem{defin*}{Definition}
\newtheorem*{theorem*}{Theorem}

\makeatletter
\newcommand{\skipitems}[1]{%
  \addtocounter{\@enumctr}{#1}%
}

\newcommand{\K}{\mathbf{K}}

\newcommand{\LS}{\operatorname{LS}}

\newcommand{\ccr}{\text{card}(R) + \aleph_0}

\newcommand{\si}{\mu_K(SI)}

\newcommand{\leap}[1]{\le_{#1}}

\newcommand{\lea}{\leap{\K}}

\title{Characterizing categoricity in several classes of modules}
\author{Marcos Mazari-Armida}
\email{Marcos.MazariA@colorado.edu}
\urladdr{https://math.colorado.edu/~mama9382/}
\address{Department of Mathematics \\ University of Colorado Boulder \\ Boulder, Colorado, USA}

\begin{document}

%%%%%%%%%%%%%%%%%%%%%%%%%%%%%%%%%%%%%%%%%%%%%%%
\begin{abstract} 

We show that the condition of being categorical in a tail of cardinals can be characterized algebraically for several classes of modules.

 \begin{theorem}Assume $R$ is an associative ring with unity. 
\begin{enumerate}
\item The class of locally pure-injective $R$-modules is $\lambda$-categorical in \emph{all} $\lambda > \text{card}(R) +\aleph_0$   if and only if $R \cong M_n(D)$ for $D$ a division ring and $n \geq 1$.
\item The class of flat $R$-modules is $\lambda$-categorical in \emph{all} $\lambda  > \text{card}(R) + \aleph_0$ if and only  if $R \cong M_n(k)$ for $k$ a local ring such that its maximal ideal is left $T$-nilpotent and $n \geq 1$.
\item Assume $R$ is a commutative ring. The class of absolutely pure $R$-modules is $\lambda$-categorical in \emph{all} $\lambda > \text{card}(R) + \aleph_0$ if and only if $R$ is a local artinian ring.
\end{enumerate}
 \end{theorem}
 
We show that in the above results it is enough to assume $\lambda$-categoricity in \emph{some} big cardinal $\lambda$. This shows that Shelah's Categoricity Conjecture holds for the class of  locally pure-injective modules, flat modules and absolutely pure modules. These classes are not first-order axiomatizable for arbitrary rings.

We provide rings such that the class of flat modules is categorical in a tail of cardinals but it is not first-order axiomatizable.

\end{abstract}

%%%%%%%%%%%%%%%%%%%%%%%%%%%%%%%%%%%%%%%%%%%%%%%

\maketitle

{\let\thefootnote\relax\footnote{{AMS 2020 Subject Classification:
Primary: 16B70, 03C60. Secondary:  16L30, 16P40, 13L05, 03C35, 03C45.
Key words and phrases.  Categoricity; Strongly indecomposable modules; Semisimple modules; Absolutely pure modules; Flat modules; Morley's Categoricity Theorem; Shelah's Categoricity Conjecture;  Abstract Elementary Classes.}}}

%\tableofcontents

%%%%%%%%%%%%%%%%%%%%%%%%%%%%%%%%%%%%%%%%%%%%%%%

\section{Introduction}

A class of structures is \emph{$\lambda$-categorical} in some cardinal $\lambda$ if there is a structure of size $\lambda$ and if any two structures of size $\lambda$ are isomorphic. A class of structures is \emph{categorical in a tail of cardinals} if there is a cardinal $\mu$ such that the class is $\lambda$-categorical in all $\lambda > \mu$. In this paper, all the structures will be left modules.

 The notion of $\lambda$-categoricity was isolated by \L o\'{s} \cite{los} in the mid 1950s. It has played a key role in the development of model theory. More precisely, the birth of modern model theory can be traced back to \emph{Morley's Categoricity Theorem} \cite{mor} which asserts that if  the class of structures of a first-order theory in a countable language is $\lambda$-categorical in \emph{some} uncountable cardinal $\lambda$ then it is $\mu$-categorical in \emph{all} uncountable cardinals $\mu$. Currently, \emph{Shelah's Categoricity Conjecture}\footnote{The statement of the conjecture is given in Section 2.3.}  \cite{sh87a}, \cite[6.13.(3)]{sh704}, which is a far-reaching generalization to Morley's Categoricity Theorem, is the main test question in the development of non-elementary model theory.\footnote{By non-elementary model theory it is understood the study of classes of structures which are not axiomatizable in first-order logic.}   It should be noted that despite many approximations and close to 3000 pages of published work towards a solution to Shelah's Categoricity Conjecture, the conjecture remains open (even the original version from the seventies).

We show that the condition of being categorical in a tail of cardinals is a natural algebraic property. We achieve this by characterizing it  algebraically for several classes of modules.

 \begin{mainm}\label{mainm} Assume $R$ is an associative ring with unity. 
\begin{enumerate}
\item (Theorem \ref{lpi})  The class of locally pure-injective $R$-modules is $\lambda$-categorical in \emph{all} $\lambda > \text{card}(R)+\aleph_0$ if and only if $R \cong M_n(D)$ for $D$ a division ring and $n \geq 1$.

\item (Theorem \ref{fmain}) The class of flat $R$-modules is $\lambda$-categorical in \emph{all} $\lambda  > \text{card}(R) + \aleph_0$ if and only  if $R \cong M_n(k)$ for $k$ a local ring such that its maximal ideal is left $T$-nilpotent and $n \geq 1$.
\item (Theorem \ref{absp-c}, Theorem \ref{linj}) Assume $R$ is a commutative ring. The following are equivalent. 
\begin{itemize}
\item The class of absolutely pure $R$-modules is $\lambda$-categorical in \emph{all} $\lambda > \text{card}(R) + \aleph_0$.
\item The class of locally injective $R$-modules is $\lambda$-categorical in \emph{all} $\lambda > \text{card}(R) + \aleph_0$.
\item $R$ is a local artinian ring.
\end{itemize}
\end{enumerate}
 \end{mainm}

The results asserted in the Main Theorem  do not rely on any model-theoretic results and instead rely heavily on algebraic results. A key result used to show the three assertions of the main theorem is the \emph{Krull-Schmidt-Remak-Azumaya Theorem}.

A natural question to ask from a model-theoretic perspective is if the results of the Main Theorem  are still true if one assumes $\lambda$-categoricity in \emph{some} big cardinal $\lambda$ instead of in \emph{all} cardinals $\lambda > \text{card}(R) +\aleph_0$.

We show that $\lambda$-categoricity in \emph{some} cardinal $\lambda > \text{card}(R) + \aleph_0$ is equivalent to categoricity in \emph{all} cardinals greater than $\text{card}(R) + \aleph_0$ for the class of flat $R$-modules using algebraic tools. In the case of absolutely pure modules, we are able to show that the same is true using a deep result from the theory of abstract elementary classes \cite[9.8]{vaseyII}. Finally, for locally pure-injective $R$-modules and  locally injective $R$-modules we show that $\lambda$-categoricity in \emph{some} cardinal $\lambda > 2^{\text{card}(R) + \aleph_0}$ is equivalent to categoricity in \emph{all} cardinals greater than $\text{card}(R) + \aleph_0$ using again \cite[9.8]{vaseyII}. Flat modules, absolutely pure modules,  locally pure-injective modules and locally injective modules are  not first-order axiomatizable for arbitrary rings, so the results for these classes do not follow from Morley's Categoricity Theorem or from Shelah's generalization to uncountable languages \cite{sh31}. Moreover, the results for these classes do not follow directly from any of the positive results known for abstract elementary classes.
 
   %We show that this is also the case for the classes of  $\Sigma$-injective and $\Sigma$-pure-injective modules.\footnote{These do not seem to be abstract elementary class for any reasonable substructure relation, so they do not provide further evidence of the veracity of Shelah Categoricity Conjecture.} 

A major contribution of the paper is to provide algebraic examples of classes categorical in a tail of cardinals. We obtain both first-order and non-first-order axiomatizable classes. For example, it follows from our results that the class of injective $R$-modules for $R$ a commutative local artinian ring, which is first-order axiomatizable, is categorical in a tail of cardinals. More importantly, we provide rings such that the class of flat modules is categorical in a tail of cardinals and it is not first-order axiomatizable (Example \ref{exflat}). This last example is one of the first algebraic examples of a class categorical in a tail of cardinals which is not first-order axiomatizable. Moreover, we give an example of an abstract elementary class which is superstable, not first-order axiomatizable and not categorical in a tail of cardinals (Example \ref{last}).

%The Main Theorem provides some evidence that categoricity in a tail of cardinals is a natural algebraic property. Nevertheless, based on the results of \cite{m2}, \cite{m3}, \cite{maz2}, we argue that superstability is a significantly more natural algebraic property. Recall that on those papers it is shown that notherian rings, perfect rings and pure-semisimple rings can be characterized via superstability. 

In an effort to make the paper concise, given a class of modules we study simultaneously if the condition of being categorical in a tail of cardinals can be characterized algebraically and its categoricity spectrum. Here the \emph{categoricity spectrum} of a class is the set of cardinals where the class is categorical. Still, we want to emphasize that both questions are somewhat independent and provide us with different information. The algebraic characterization of categoricity in a tail of cardinals displays the naturality of categoricity as an  algebraic notion and can be used to provide many examples of classes categorical in a tail of cardinals. On the other hand, the study of the categoricity spectrum done in this paper provides further evidence of the veracity of Shelah's Categoricity Conjecture and showcases how abstract elementary classes results can be used to study specific classes of structures.

It is worth mentioning that this is not the first paper where categoricity in a tail of cardinals is charactetized algebraically for classes of algebraic structures. In \cite{givant}, \cite{givant2} and \cite{pal},  categoricity in a tail of cardinals is characterized algebraically for classes of algebraic structures that are first-order axiomatizable by universal Horn sentences. In particular, the result for the class of all $R$-modules presented in this paper (Lemma \ref{mmod}) can be obtained using\cite[Theorem I]{givant}. 

%It is widely accepted classes of structures categorical in a tail of cardinals are virtous. In this paper we give a series of examples of natural classes of modules which are categorical in a tail of cardinals. 

The paper is divided into four sections. Section 2 presents some basic results and preliminaries. Section 3 has the main results of the paper. Categoricity in a tail of cardinals is  characterized algebraically for several classes of modules. It is shown that Shelah's Categoricity Conjecture holds for the class of locally pure-injective modules, absolutely pure modules, locally injective modules and flat modules. Section 4 has some open problems.

After showing a preliminary version of this paper to Sebastien Vasey, he realized that it is possible to show that Shelah's Categoricity Conjecture is true for the class of locally pure-injective modules, flat modules, absolutely pure modules and locally injective modules using only model-theoretic tools. His argument relies on some known instances of Shelah's Categoricity Conjecture\footnote{More precisely, to obtain the result for locally pure-injective modules, absolutely pure modules and locally injective modules one uses the result for homogeneous classes and to obtain the result for flat modules one uses the result for excellent classes.} together with stronger results from the theory of AECs of modules. It is worth emphasizing that his argument requires some additional work before being able to apply the known instances of Shelah's Categoricity Conjecture. 

Some of the results of this paper were obtained while the author was working on a Ph.D. under the direction of Rami Grossberg at Carnegie Mellon University and I would like to thank Professor Grossberg for his guidance and assistance in my research in general and in this work in particular. I would like to thank John T. Baldwin,  Samson Leung, Thomas Kucera, Philipp Rothmaler, Daniel Simson, Agnes Szendrei, Sebastien Vasey and two anonymous referee for several comments that helped improve the paper. In particular, I would like to thank an anonymous referee for pointing out Example \ref{exflat}.(3).

\section{Basic results and preliminaries}

In this section we begin by introducing some module theoretic preliminaries and a module theoretic framework where we can transfer categoricity. Then we analyse the class of $R$-modules. We finish by briefly recalling the notions of abstract elementary classes that are used in this paper. It is worth emphasizing that if the reader is only interested in the algebraic characterization of being categorical in a tail of cardinals she or he can skip the AECs preliminaries as they are not used for those results.

\subsection{Basic results}   All rings considered in this paper are associative with unity and all modules are left modules. Given a module $M$, we will write $|M|$ for its underlying set and $\| M \|$ for its cardinality. For a ring $R$, we will write $\text{card}(R)$ for its cardinality.

We say that a class of modules is \emph{$\lambda$-categorical} in some cardinal $\lambda$ if there is a module of size $\lambda$ in the class and if any two modules in the class of size $\lambda$ are isomorphic. A class of modules is categorical \emph{in a tail of cardinals} if there is a cardinal $\mu$ such that the class is $\lambda$-categorical in all $\lambda > \mu$.\footnote{From a model-theoretic perspective, all the classes studied in this paper have as their language the standard language of modules, i.e., for a ring $R$ we take $L_{R}= \{0, +,-\} \cup \{ r\cdot  : r \in R \}$. Two modules are isomorphic as structures in the language $L_R$ if and only if they are isomorphic as modules.}

A ring $R$ is \emph{local} if the sum of two non-invertible elements of $R$ is non-invertible. We will be specially interested in strongly indecomposable modules.

\begin{defin}
A left $R$-module $M$ is \emph{strongly indecomposable} if $M$ is non-zero and the endomorphism ring of $M$ is local. Strongly indecomposable modules are sometimes called endolocal modules.
\end{defin}

A module is \emph{indecomposable} if $M$ is non-zero and $M$ cannot be written as the direct sum of two non-trivial submodules.

\begin{remark}\label{ermk}\
\begin{itemize}
\item  If a module is strongly indecomposable, then it is indecomposable. The other direction  fails in general, e.g., $R=M = \mathbb{Z}$.
\item If an injective module is indecomposable, then it is strongly indecomposable.
\end{itemize}
\end{remark} 

The next result is the key module theoretic result that is used in the paper. It assures us that if a module can be decomposed into a direct sum of strongly indecomposable modules, then the decomposition is basically unique. 
 
\begin{fact}\label{strong}(Krull-Schmidt-Remak-Azumaya Theorem, see e.g. \cite[2.13]{fac}) Let $M \cong \bigoplus_{i \in I} M_i \cong \bigoplus_{j \in J} N_j$. If $M_i, N_j$ are strongly indecomposable for every $i \in I$ and $j \in J$, then there is $\sigma: I \to J$ a bijective function such that for every $i \in I$, $M_i \cong N_{\sigma(i)}$. 
\end{fact}

Let us introduce a framework where we will show categoricity can be transferred from one cardinal to many cardinals. 

\begin{hypothesis}\label{hyp1}
$K$ is a class of left $R$-modules such that the following hold:
\begin{enumerate}
\item Every non-zero module in $K$ can be written as a direct sum of strongly indecomposable modules in $K$.
\item There is a cardinal $\mu$ such that if $M \in K$ and $M$ is a strongly indecomposable module then $\| M \| \leq \mu$. We denote by $\si$ the minimum $\mu$ satisfying such a property.
\item If $M \in K$ and $\theta$ is a cardinal, then $M^{(\theta)} \in K$ where $M^{(\theta)}$ denotes the direct sum of $\theta$ copies of $M$.
\end{enumerate}
\end{hypothesis}

\begin{remark}
We will explicitly mention when we use Hypothesis \ref{hyp1}. 
\end{remark}

We introduce examples of classes of modules satisfying Hypothesis \ref{hyp1}. We will introduce an additional example in Lemma \ref{cor-ab}.

\begin{example}\label{ex!}\
\begin{enumerate}
\item $\Sigma$-injective $R$-modules. A module $M$ is $\Sigma$-injective if $M^{(I)}$ is injective for every set $I$. Condition (1) holds by \cite[Theo 1]{new} and the fact that injective indecomposable modules are strongly indecomposable (see e.g. \cite[2.12]{fac}).  $\si \leq \ccr$ by the closure of $\Sigma$-pure-injective modules under pure submodules. It is clear that condition (3) holds.
\item $\Sigma$-pure-injective $R$-modules. A module $M$ is $\Sigma$-pure-injective if $M^{(\aleph_0)}$ is pure-injective. Condition (1) is known to hold (see e.g.\cite[2.29]{fac}). $\si \leq \ccr$ by the closure of $\Sigma$-pure-injective modules under pure submodules.  It is clear that condition (3) holds.
\item Semisimple $R$-modules. A module $M$ is semisimple if it is the direct sum of simple modules, where a module is simple if  it does not have non-zero proper submodules. Condition (1) holds by the fact that simple modules are strongly indecomposable. It is clear that $\si \leq \ccr$ and that (3) holds.
\end{enumerate}
\end{example}

\begin{lemma}\label{basicl} Let $K$ be a class of modules satisfying Hypothesis \ref{hyp1}. The following are equivalent.
\begin{enumerate}
\item $K$ is $\lambda$-categorical in \emph{all} $\lambda > \si$.
\item $K$ is $\lambda$-categorical in \emph{some} $\lambda > \si$.
\item There is a unique strongly indecomposable module in $K$ up to isomorphisms.
\item There is a strongly indecomposable module $U$ in $K$ such that every non-zero module  in $K$ is isomorphic to $U^{(\lambda)}$ for some cardinal $\lambda$. 
\end{enumerate}
\end{lemma}
\begin{proof}
$(1) \Rightarrow (2)$: Clear.

$(2) \Rightarrow (3)$:  Let $\lambda$ be such that $K$ is $\lambda$-categorical and  $\lambda > \si$. Assume for the sake of contradiction that there are $M, N$ non-isomorphic strongly indecomposable modules in $K$. By the hypothesis on $K$, $\| M \|, \| N \| \leq \si$ and $M^{(\lambda)}, N^{(\lambda)} \in K$ of cardinality $\lambda$, then $M^{(\lambda)} \cong N^{(\lambda)}$ by $\lambda$-categoricity. Therefore, $M$ and $N$ are isomorphic by Fact \ref{strong}, which gives us a contradiction.

$(3) \Rightarrow (4)$: Let $U$ be the unique strongly indecomposable module in $K$ and $M$ a non-zero module in $K$. 
It follows from Condition (1) of Hypothesis \ref{hyp1} that $M \cong U^{(\lambda )}$ for some $\lambda$. 

$(4) \Rightarrow (1)$: Let $\lambda > \si$ and $U$ be a strongly indecomposable module in $K$ given by (4). Let $M, N \in K$ of cardinality $\lambda$. By (4), $M \cong U^{(\lambda_1 )}$
and $N \cong U^{(\lambda_2 )}$ for $\lambda_1, \lambda_2$ cardinals. Since $\| U \| < \lambda$, $\lambda_1 = \lambda_2$ and hence $M \cong N$. \end{proof}

We finish by recording two corollaries.

\begin{cor}\label{cors} Let $K$ be the class of $\Sigma$-injective $R$-modules or the class of $\Sigma$-pure-injective $R$-modules or the class of semisimple $R$-modules . The following are equivalent.
\begin{enumerate}
\item $K$ is $\lambda$-categorical in \emph{all} $\lambda > \ccr$.
\item $K$ is $\lambda$-categorical in \emph{some} $\lambda > \ccr$.
\end{enumerate}
\end{cor}

 \begin{remark}
Later we will show that categoricity in a tail of cardinals can be characterized algebraically for some classes of modules. A natural question to ask is if the same is true for $\Sigma$-injective modules, $\Sigma$-pure-injective modules or semisimple modules (see Question \ref{q1}).
 \end{remark}
 
 \subsection{$R$-modules} In this subsection we study the class of left $R$-modules.  The results of this subsection are not new. The reason we decided to present the results  is because the techniques used to analyze the class of $R$-modules are similar to those of our main results and  because we could not find a purely algebraic treatment of the results presented in this subsection.
 
 We denote by $R$-Mod the class of \emph{all left $R$-modules}. A ring $R$ is \emph{semisimple} if every left $R$-module is projective.\footnote{A rings $R$ is semisimple if and only if $J(R)=0$ and $R$ is artinian \cite[4.14]{lam1}} A module is \emph{projective} if it is a direct summand of a free module. The next result is the key result regarding semisimple rings.

\begin{fact}\label{wa} (Wedderburn-Artin Theorem, see e.g. \cite[3.5]{lam1})
If $R$ is semisimple, then $R \cong  M_{n_1}(D_1) \times \cdots \times M_{n_{m}}(D_{m})$ for some $m \in \mathbb{N}$, $D_1, \cdots, D_m$ division rings and $n_1, \cdots, n_m \in \mathbb{N}$. This decomposition is unique. Moreover, there are exactly $m$ strongly indecomposable modules (up to isomorphisms) given by $V_{D_1}, \cdots, V_{D_m}$ corresponding to $M_{n_1}(D_1), \cdots, M_{n_m}(D_m)$ respectively where $V_{D_i}^{n_i} \cong M_{n_i}(D_i)$ for every $i$.
\end{fact} 

We show that if a ring is semisimple, then the class of modules satisfies Hypothesis \ref{hyp1}.

\begin{lemma}\label{cor-mod}
 If $R$ is semisimple, then the class of left $R$-modules satisfies Hypothesis \ref{hyp1} with $\si \leq \text{card}(R) +\aleph_0$.
 \end{lemma}
 \begin{proof} Since $R$ is a semisimple ring,  every left $R$-module is semisimple (see e.g \cite[1.2]{fac}). 
 Then the assertion follows directly from Example \ref{ex!}.
 \end{proof}

%\begin{remark}\label{base}\
%\begin{itemize}
%\item If $R$ is left semisimple then a module is simple if an only if it is strongly indecomposable.
%\item Given a divisible ring $D$, $M_{n}(D)$ has a unique simple module (up to isomorphisms) $V_D$ such that $V_D^n \cong M_{n}(D)$.

%\end{itemize}
%\end{remark}

\begin{lemma}\label{mmod} Assume $R$ is an associative ring with unity. The following are equivalent.
\begin{enumerate}
\item $R\text{-Mod}$ is $\lambda$-categorical in all $\lambda > \text{card}(R) + \aleph_0$.
\item $R\text{-Mod}$ is $\lambda$-categorical in some $\lambda > \text{card}(R) + \aleph_0$.
\item $R \cong M_n(D)$ for $D$ a division ring and $n \geq 1$.
\end{enumerate}
\end{lemma} 
\begin{proof}
We only need to show (2) implies (3) and (3) implies (1).

$(2) \Rightarrow (3)$: We show first that $R$ is a semisimple ring. It is enough to show that every finitely generated left $R$-modules is projective by \cite[2.8]{lam1}.  Let $M$ be a non-zero finitely generated left $R$-module, then $\| M \| \leq \text{card}(R) + \aleph_0$. Since $\lambda > \text{card}(R) + \aleph_0$, $M^{(\lambda)}$ is free by $\lambda$-categoricity and the fact that there is a free module of cardinality $\lambda$. Therefore, $M$ is projective, so $R$ is semisimple.

Since $R$ is semisimple, it follows from Fact \ref{wa} that $R \cong  M_{n_1}(D_1) \times \cdots \times M_{n_{m}}(D_{m})$ for some $m \in \mathbb{N}$, $D_1, \cdots, D_m$ division rings and $n_1, \cdots, n_m \in \mathbb{N}$. By Lemma \ref{cor-mod} and Lemma \ref{basicl} it follows that $R$-Mod has a unique strongly indecomposable module. Therefore, $m=1$  by Fact \ref{wa}.

$(3) \Rightarrow (1)$: Since $R$ is semisimple and  has a unique strongly indecomposable module by Fact \ref{wa}, the results follows from Lemma \ref{cor-mod} and Lemma \ref{basicl}. \end{proof}

\begin{remark}
The equivalence between (1) and (2) follows directly from Morley's Categoricity Theorem for countable rings and from \cite{sh31} for uncountable rings, the reason we did not simply quote Morley's Categoricity Theorem or Shelah's generalization is to provide an algebraic argument on how to transfer categoricity. 
\end{remark}

\begin{remark}
The direction (3) implies (1) is known to hold, see for example the last paragraph of \cite{piro}. The direction (1) implies (3) can be obtained using \cite[Theorem I]{givant} and is folklore. 
\end{remark}

For commutative rings the above result can be improved to the following.

\begin{cor}\label{mod-c} Assume $R$ is a commutative ring with unity. The following are equivalent.
\begin{enumerate}
\item $R\text{-Mod}$ is $\lambda$-categorical in some (all) $\lambda > \text{card}(R) + \aleph_0$.
\item $R$ is a field.
\item All left $R$-modules are free.
\end{enumerate}
\end{cor}

For finite rings we can also get the following improvement.

\begin{cor}\label{finte} Assume $R$ is an associative and finite ring with unity. The following are equivalent.
\begin{enumerate}
\item $R\text{-Mod}$ is $\lambda$-categorical in some (all) $\lambda >  \aleph_0$.
\item  $R \cong M_n(D)$ for $D$ a field and $n \geq 1$.
\item All infinite left $R$-modules are free.
\item  $R\text{-Mod}$ is $\lambda$-categorical in all $\lambda \geq  \aleph_0$.
\end{enumerate}
\end{cor} 
\begin{proof}
$(1) \Rightarrow (2)$: Follows from Theorem \ref{mmod} and Wedderburn little theorem that finite division rings are fields.

$(2) \Rightarrow (3)$: Let $M$ be an infinite $R$-module of cardinality $\lambda$. Let $V_D$ be the unique strongly indecomposable module of $R$. It is easy to see that $R \cong V_D^n$ and that $M \cong V_D^{(\lambda)}$. Observe that $V_D^{(\lambda)} \cong (V_D^n)^{(\lambda)}$, so $M \cong R^{(\lambda)}$. Therefore, $M$ is free.

$(3) \Rightarrow (4)$: Any two free modules of the same cardinality are isomorphic.

$(4) \Rightarrow (1)$: Clear.  \end{proof}

\begin{remark} $(1)$ if and only if $(4)$ follows directly from the Main Theorem of \cite{bala}.
\end{remark} 

\subsection{Abstract elementary classes} \emph{Abstract elementary classes} (AECs for short) were introduced  by Shelah in \cite{sh88}. An \emph{AEC} is a pair $\K=(K \lea)$ where $K$ is a class of structures and $\lea$ is a partial order on $K$ extending the substructure relation such that $\K$ is closed under direct limits and satisfies the coherence property and an instance of the Downward L\"{o}wenheim-Skolem theorem. The reader can consult the definition in \cite[4.1]{baldwinbook09}. 

The main test question in the development of the Classification Theory for Abstract Elementary Classes is the  following far-reaching conjecture  asserted by Shelah in \cite[6.13.(3)]{sh704}.\footnote{Instances of Shelah's Categoricity Conjecture date back to the mid-senventies \cite{sh87a}.}
  % In this paper we will focus on classes of modules, so isomorphism can be understood as isomorphisms in the sense of module theory. 

\begin{sec-1} Assume $\K$ is an abstract elementary class.
If $\K$ is $\lambda$-categorical in \emph{some} cardinal $\lambda \geq \beth_{(2^{\LS(\K)})^+}$, then $\K$ is $\mu$-categorical in \emph{all} cardinals $\mu \geq \beth_{(2^{\LS(\K)})^+}$.\footnote{Recall that $\beth_0 = \aleph_0$, $\beth_{\alpha + 1} = 2^{\beth_\alpha}$ and $\beth_\delta = sup_{\alpha < \delta} \beth_\alpha$ for $\delta$ a limit ordinal.}
\end{sec-1}  

 Despite many approximations and close to 3000 pages of published work towards a solution to Shelah's Categoricity Conjecture, the conjecture remains open (even the original version from the seventies for $L_{\omega_1, \omega}$). In this paper, we will use one such an approximation due to Vasey \cite[9.8]{vaseyII}, but before we introduce his result we present a few more preliminaries. 
 
 There is a semantic notion of type on AECs called \emph{Galois-types}, see for example \cite[8.7]{baldwinbook09} for the definition. Galois-types are equivalence classes over 
certain triples, so they might not be determined by their restrictions to smaller subsets, when this happens we say that the AEC is \emph{tame}. More precisely, $\K$ is \emph{$\LS(\K)$-tame} if for any $M \in \K$ and $p, q$ Galois-types (of length one), if $p \neq q$ then there is $N \lea M$ such that $p\upharpoonright{N} \neq q\upharpoonright{N}$ and $\| N \| = \LS(\K)$. Finally, an AEC has the amalgamation property if any span $M \lea N_1, N_2$ can be completed to a commutative square.

\begin{fact}[{\cite[9.8]{vaseyII}}]\label{sebsc}Assume $\K$ is an abstract elementary class which is  $LS(\K)$-tame and has arbitrarily large models and the amalgamation property. If $\K$ is $\lambda$-categorical in some $\lambda > LS(\K)$, then $\K$ is $\lambda$-categorical in all cardinals of the form $\beth_\delta$ where $\delta = (2^{\LS(\K)})^+ \alpha $ (ordinal multiplication) for some ordinal $\alpha > 0$.
\end{fact} 

\begin{remark}
Vasey's result is obtained using sophisticated machinery from AECs such as good $\lambda$-frames. In this paper, we will simply apply the result to some AECs of modules which are known to satisfy the hypothesis of the theorem by earlier work on abstract elementary classes of modules. 
\end{remark}

This is all the theory of AECs that we will use in this paper. An introduction to AECs from an algebraic perspective is given in \cite{m4}. Classical texts on AECs include \cite{baldwinbook09},  \cite{grossberg2002} and \cite{shelahaecbook}.

\section{Main results} 

In this section we study the categoricity spectrum in the class of locally pure-injective modules, absolutely pure modules, and flat modules. We characterize algebraically when these classes are categorical in a tail of cardinals.

\subsection{Locally pure-injective modules} In this subsection we study the class of locally pure-injective modules. Recall that $M$ is a \emph{pure submodule} of $N$, denoted by $M \leq_{p} N$, if $M$ is a submodule of $N$ and for every $L$ right $R$-module $L \otimes M \to L \otimes N$ is a monomorphism. Equivalently $M$ is a submodule of $N$ and positive primitive formulas (existentially quantified systems of linear equations) are preserved between $M$ and $N$.

\begin{defin} A left $R$-module $M$ is \emph{locally pure-injective} if for every finite subset of $M$ there is a pure-injective pure submodule of $M$ containing it. We denote the class of locally pure-injective  modules by $R\text{-l-pi}$.
\end{defin}

A ring is \emph{left pure-semisimple} if every left module is pure-injective. It is clear that if a ring is pure-semisimple then the class of locally pure-injective modules is the same as the class of modules.  

\begin{lemma}\label{cor-lpi}
 If $R$ is left pure-semisimple, then the class of locally pure-injective $R$-modules satisfies Hypothesis \ref{hyp1} with $\si \leq \text{card}(R) +\aleph_0$.
 \end{lemma}
 \begin{proof} Since $R$ is a left pure-semisimple ring, every locally pure-injective $R$-module is $\Sigma$-pure-injective. Then the assertion follows directly from Example \ref{ex!}.\end{proof} 
 
In order to use Fact \ref{sebsc} we will need the following model-theoretic result.

\begin{fact}[{\cite[3.3, 3.5, 3.10]{maz2}}]\label{lpi-easy}
$\K= (R\text{-l-pi}, \leq_p)$ is an abstract elementary class with $\text{card}(R) + \aleph_0 \leq \LS(\K) \leq 2^{\ccr}$ that is $\LS(\K)$-tame and has arbitrarily large models and the amalgamation property. 
\end{fact}

\begin{theorem}\label{lpi} Assume $R$ is an associative ring with unity. The following are equivalent.
\begin{enumerate}
\item $R\text{-l-pi}$ is $\lambda$-categorical in all $\lambda > \text{card}(R) + \aleph_0$.
\item $R\text{-l-pi}$ is $\lambda$-categorical in some $\lambda > 2^{\text{card}(R) + \aleph_0}$.
\item $R \cong M_n(D)$ for $D$ a division ring and $n \geq 1$.
\end{enumerate}
\end{theorem} 
\begin{proof} We only need to show (2) implies (3) and (3) implies (1).

$(2) \Rightarrow (3)$: Let $\K=(R\text{-l-pi}, \leq_p)$.  Since $\K$ is an  abstract elementary class that is $\LS(\K)$-tame and has arbitrarily large models and the amalgamation property by Fact \ref{lpi-easy},  $R$-l-pi is  $\beth_\delta$-categorical in every $\delta$ divisible by $(2^{\LS(\K)})^+$ by Fact \ref{sebsc}.

We show that $R$ is left pure-semisimple. Let $M$ be a non-zero left $R$-modules and $PE(M)$ be its pure-injective envelope.  Let $\mu =\beth_\delta$ where $\delta = (2^{\LS(\K)})^+  (\| PE(M)  \| + 1)$ (taking ordinal product).  Let $N$ be a locally pure-injective module of size $\mu$, then $\| PE(N) \| \leq \beth_\delta^{\text{card}(R) + \aleph_0}=\beth_\delta=\mu$. The inequality follows from \cite[3.11]{ziegler} as we are taking the pure-injective envelope and the equality from the fact that $cf(\beth_\delta)= (2^{\LS(\K)})^+  > \ccr$.  Hence $PE(N)$  has cardinality $\mu$. Moreover, $(PE(M)^{(\mu)})^{(\aleph_0)}$ is locally pure-injective because locally pure-injective modules are closed under direct sums \cite[2.4]{zimm2}. As  $\|  PE(M) \| < \mu$, $(PE(M)^{(\mu)})^{(\aleph_0)} \cong PE(N)$ by $\mu$-categoricity. Hence $PE(M)^{(\mu)}$ is $\Sigma$-pure-injective. Therefore, $M$ is  pure-injective as $\Sigma$-pure-injective modules are closed under pure submodules and $M \leq_p PE(M)$. Therefore, $R$ is left pure-semisimple.

Since $R$ is left pure-semisimple then the class of locally pure-injective $R$-modules is the same as the class of $R$-modules. Then the result follows directly from Lemma \ref{mmod}.

$(3) \Rightarrow (1)$: Since $R \cong M_n(D)$, $R$ is semisimple. Hence $R$ is left pure-semisimple. Then the result follows directly from Lemma \ref{mmod}. \end{proof}

\begin{remark}
The equivalence between (1) and (3) can be obtained without using Fact \ref{sebsc}. The proof of (1) implies (3) can be done as (2) implies (3) but there is no need to use  Fact \ref{sebsc} as by assumption the class is $\lambda$-categorical in all $\lambda > \text{card}(R) + \aleph_0$.
\end{remark}

\begin{remark} If there is an infinite left $R$-module which is not locally pure-injective, then 
the class of locally pure-injective modules is not first-order axiomatizable. The assertion follows from the fact that $M$ is an elementary substructure of $PE(M)$ for every module $M$ by \cite{sab}. An example of such a ring is $\mathbb{Z}$ \cite[2.10]{zimm2}. Therefore, the equivalence between (1) and (2) of Theorem \ref{lpi} is an instance of Shelah's Categoricity Conjecture which does not follow from Morley's Categoricity Theorem or Shelah's generalization. Moreover, the result does not follow directly from any of the positive results known for AECs.
\end{remark}

\begin{remark}
A natural question to ask is if the categoricity threshold can be lowered to $\ccr$ in the above theorem. 
\end{remark}

\subsection{Absolutely pure modules}  In this subsection we study the class of absolutely pure modules.  

\begin{defin} A left $R$-module $M$ is \emph{absolutely pure} if for every $N$, if $M$ is a submodule of $N$ then $M$ is a pure submodule of $N$. We denote the class of all absolutely pure $R$-modules by $R$-AbsP.
 \end{defin}
 
 Recall that $M$ is absolutely pure  if and only if $Ext_R^1(N, M) =0$ for every finitely presented module $N$. It is clear that injective modules are absolutely pure. We show that if a ring is left noetherian, then the class of absolutely pure modules satisfies Hypothesis \ref{hyp1}.

% \begin{fact}\label{absp-b} Assume $R$ is an associative ring with unity.
% \begin{enumerate}
 %\item (\cite[2.12]{fac}) An injective module is indecomposable if and only if it is strongly indecomposable. 
 %\item (\cite[4.4.17]{prest09}) $R$ is left noetherian if and only if every absolutely pure left $R$-module is injective.
% \item (\cite[3.48]{lam2}) If $R$ is left noetherian, then every absolutely pure (injective) left $R$-modules is a direct sum of indecomposable injective modules.
 %\end{enumerate}
% \end{fact}

 \begin{lemma}\label{cor-ab}
 If $R$ is left noetherian, then the class of absolutely pure $R$-modules satisfies Hypothesis \ref{hyp1} with $\si \leq \text{card}(R) +\aleph_0$.
 \end{lemma}
 \begin{proof}
 Since $R$ is left noetherian, every absolutely pure module is $\Sigma$-injective by \cite[4.4.17]{prest09}. Then the assertion follows directly from Example \ref{ex!}.
 \end{proof}
 
In order to use Fact \ref{sebsc} we will need the following model-theoretic result.

\begin{fact}[{\cite[3.3, 3.5, 3.10]{maz2}}]\label{abspb}
$\K= (R\text{-AbsP}, \leq_p)$ is an abstract elementary class with $\LS(\K) = \ccr$ that is $\LS(\K)$-tame and has arbitrarily large models and the amalgamation property. 
\end{fact}

\begin{theorem}\label{absp} Assume $R$ is an associative ring with unity. The following are equivalent.
\begin{enumerate}
\item $R\text{-AbsP}$ is $\lambda$-categorical in all $\lambda > \text{card}(R) + \aleph_0$.
\item $R\text{-AbsP}$ is $\lambda$-categorical in some $\lambda > \text{card}(R) + \aleph_0$.
\item $R$ is left noetherian and there is a unique indecomposable injective module up to isomorphisms.

\end{enumerate}
\end{theorem}
\begin{proof}
We only need to show (2) implies (3) and (3) implies (1).

$(2) \Rightarrow (3)$: Let $\K=(R\text{-AbsP}, \leq_p)$.  Since $\K= (R\text{-AbsP}, \leq_p)$ is an  abstract elementary class that is $\LS(\K)$-tame and has arbitrarily large models and the amalgamation property by Fact \ref{abspb},  $R$-AbsP is  $\beth_\delta$-categorical in every $\delta$ divisible by $(2^{\ccr})^+$ by Fact \ref{sebsc}.

We show that $R$ is left noetherian by showing that every absolutely pure left $R$-modules is injective, this is enough by \cite[4.4.17]{prest09}. Let $M$ be a non-zero absolutely pure left $R$-modules and $\mu =\beth_\delta$ where $\delta = (2^{\ccr})^+  (\| M \| + 1)$ (taking ordinal product).  Let $N$ be an absolutely pure module of size $\mu$, then $\| E(N) \| \leq \beth_\delta^{\text{card}(R) + \aleph_0}=\beth_\delta=\mu$. The inequality follows from \cite[Theo 1]{eklof} as we are taking the injective envelope and the equality from the fact that $cf(\beth_\delta) = (2^{\ccr})^+$. Hence $E(N)$  has cardinality $\mu$. Moreover, $M^{(\mu)}$ is absolutely pure because absolutely pure modules are closed under direct sums \cite[2.3.5]{prest09}. As  $\|  M \| <\mu$, $M^{(\mu)} \cong E(N)$ by $\mu$-categoricity. Hence $M^{(\mu)}$ is injective. Therefore, $M$ is injective as injective modules are closed under direct summands. Therefore, $R$ is left noetherian.

The existence of a unique indecomposable module follows directly from Lemma \ref{cor-ab}, Lemma \ref{basicl} and Remark \ref{ermk}.

$(3) \Rightarrow (1)$: The result follows directly from Lemma \ref{cor-ab} and Lemma \ref{basicl} using that indecomposable injective modules are strongly indecomposable. \end{proof}

\begin{remark}
The equivalence between (1) and (3) can be obtained without using Fact \ref{sebsc}. The proof of (1) implies (3) can be done as (2) implies (3) but there is no need to use  Fact \ref{sebsc} as by assumption the class is $\lambda$-categorical in all $\lambda > \text{card}(R) + \aleph_0$.
\end{remark}

\begin{remark}
The equivalence between (1) and (2) of Theorem \ref{absp} is an instance of Shelah's Categoricity Conjecture. Since the class of absolutely pure left $R$-modules is first-order axiomatizable if and only if $R$ is left coherent \cite[3.4.24]{prest09}, the result does not follow from Morley's Categoricity Theorem or Shelah's generalization. Moreover, the result does not follow directly from any of the positive results known for AECs.
\end{remark}

\begin{remark}
 It is worth mentioning that (2) implies that the class of absolutely pure left $R$-modules is first-order axiomatizable, but this is a by-product of (2) implies (3). This is a weaker statement than that of  \cite[3.25]{maz2}.
\end{remark}

We are able to obtain a significantly nicer algebraic characterization if the ring is commutative. In order to do that we will need a few more algebraic results.

\begin{fact}[{\cite[3.62]{lam2}}]\label{bij}
Assume $R$ is a commutative noetherian ring with unity. There is a bijection between the set of isomorphism classes of injective indecomposable $R$-modules and the prime ideals of $R$.
\end{fact}

Recall that the Krull dimension of $R$ is the supremum of the lengths of all chains of prime ideals
in $R$, we denote it by $dim(R)$.

\begin{fact}[{\cite[8.5]{atmc}}]\label{art=noe} Assume $R$ is a commutative ring with unity.
$R$ is artinian if and only if $R$ is noetherian and $dim(R)=0$. 
\end{fact}

Assuming $R$ is commutative we improve Theorem \ref{absp}.

\begin{theorem}\label{absp-c} Assume $R$ is a commutative ring with unity. The following are equivalent.
\begin{enumerate}
\item $R\text{-AbsP}$ is $\lambda$-categorical in all $\lambda > \text{card}(R) + \aleph_0$
\item $R\text{-AbsP}$ is $\lambda$-categorical in some $\lambda > \text{card}(R) + \aleph_0$
\item $R$ is a local artinian ring.
\end{enumerate}
\end{theorem}
\begin{proof} We only need to show (2) implies (3) and (3) implies (1).  

$(2) \Rightarrow (3)$: It follows from Theorem \ref{absp} that $R$ is noetherian and that there is a unique indecomposable injective $R$-module. Then it follows from Fact \ref{bij} that $R$ has a unique prime ideal. Since every maximal ideal is prime, $R$ is a local ring (see e.g. \cite[1.10]{fac}). Since there is a unique prime ideal clearly $dim(R)=0$, then $R$ is  artinian by Fact \ref{art=noe}. Hence $R$ is a local artinian ring.

$(3) \Rightarrow (1)$: Since $R$ is artinian, $R$ is noetherian. Moreover, since $R$ is commutative artinian every prime ideal is maximal \cite[8.1]{atmc}. Hence there is a unique prime ideal because $R$ is local.  Since $R$ is a commutative noetherian ring, it follows that there is a unique indecomposable injective $R$-module by Fact \ref{bij}. Therefore, $R\text{-AbsP}$ is $\lambda$-categorical in all $\lambda > \text{card}(R) + \aleph_0$ by Theorem \ref{absp}. \end{proof}

\begin{example}
Let $R$ be  commutative local artinian ring that is not a field. By Corollary \ref{mod-c} and Theorem \ref{absp-c}, $R\text{-Mod}$ and $R\text{-AbsP}$ are different. Moreover, $R\text{-AbsP}$ is first-order axiomatizable since $R$ is coherent \cite[3.4.24]{prest09}. Therefore, commutative local artinian rings that are not fields provide a new natural family of examples of classes axiomatizable in first-order which are categorical in a tail of cardinals. This is a natural family of rings as every commutative artinian ring is a  finite product of commutative local artinian rings \cite[23.12]{lam1}. An explicit example of such a ring is $\mathbb{R}[x]/(x^2)$. 
\end{example}

We finish this subsection by showing that the results we obtained for absolutely pure modules can be extended to locally injective modules (also called finitely injective modules). A left $R$-module $M$ is \emph{locally injective} if for every finite subset of $M$ there is an injective submodule of $M$ containing it. We denote the class of locally injective $R$-modules by $R\text{-l-inj}$.

\begin{theorem}\label{linj} Assume $R$ is an associative ring with unity. The following are equivalent.
\begin{enumerate}
\item $R\text{-l-inj}$ is $\lambda$-categorical in all $\lambda > \text{card}(R) + \aleph_0$.
\item $R\text{-l-inj}$ is $\lambda$-categorical in some $\lambda > 2^{\text{card}(R) + \aleph_0}$.
\item $R$ is left noetherian and there is a unique indecomposable injective module up to isomorphisms.

\end{enumerate}
\end{theorem}
\begin{proof}
We only need to show (2) implies (3) and (3) implies (1). 

$(2) \Rightarrow (3)$: We show that $R$ is left noetherian by showing that every injective left $R$-modules is $\Sigma$-injective, this is enough by \cite[4.4.17]{prest09}. 

Let $\K=(R\text{-l-inj}, \leq_p)$.   $\K$ is an  abstract elementary class with $\text{card}(R) + \aleph_0 \leq \LS(\K) \leq 2^{\ccr}$ that is $\LS(\K)$-tame and has arbitrarily large models and the amalgamation property by \cite[3.3, 3.5, 3.10]{maz2},  then $R$-l-inj is  $\beth_\delta$-categorical in every $\delta$ divisible by $(2^{\LS(\K)})^+$ by Fact \ref{sebsc}. Then doing a similar argument to that of $(2)$ implies $(3)$ of Theorem \ref{absp}, by using that locally injective modules are absolutely pure, it follows that every  injective module is $\Sigma$-injective. Hence $R$ is left noetherian.  For left notherian rings, the class of locally injective modules is the same as the class of absolutely pure modules. Then the result follows from Theorem \ref{absp-c}.

$(3) \Rightarrow (1)$: Since $R$ is left noeherian, the class of locally injective $R$-modules is the same as the class of absolutely pure $R$-modules. Then the result follows from Theorem \ref{absp-c}. \end{proof}

\begin{remark}
If $R$ is not left coherent, then $R\text{-l-inj}$ is not first-order axiomatizable. The assertion can be proved using the same argument as that of \cite[3.4.24]{prest09} and noticing that locally injective modules are absolutely pure. Therefore, the equivalence between (1) and (2) of Theorem \ref{linj} is an instance of Shelah's Categoricity Conjecture which does not follow from Morley's Categoricity Theorem or Shelah's generalization. Moreover,  the result does not follow directly from any of the positive results known for AECs.
\end{remark}

\begin{remark}
A natural question to ask is if the categoricity threshold can be lowered to $\ccr$ in the above theorem. 
\end{remark}

\begin{remark}  Let $R$ be a commutative ring with unity. It is clear that: $R\text{-l-inj}$ is $\lambda$-categorical in all $\lambda > \text{card}(R) + \aleph_0$ if and only if $R$ is a local artinian ring.
\end{remark}

\subsection{Flat modules}  In this subsection we study the class of flat left $R$-modules.

\begin{defin}
A left $R$-module $F$ is \emph{flat} if $(-) \otimes F$ is an exact functor. We denote by $R$-Flat the class of flat left $R$-modules.
\end{defin}

 A ring $R$ is \emph{left perfect} if every countably generated flat left $R$-module is projective.The standard definition of left perfect ring is that every flat left $R$-module is projective, but it is enough to consider only countably generated flat modules by \cite[24.24]{lam1} and the proof of \cite[24.25]{lam1}. One can show that $R$ is left perfect if and only if $R/J(R)$ is semisimple and $J(R)$ is left $T$-nilpotent \cite[24.25]{lam1}. Here $J(R)$ denotes the \emph{Jacobson radical}, i.e., the intersections of all maximal left ideals. An ideal $I$ is \emph{left $T$-nilpotent} if for every $\{a_n \}_{n < \omega} \subseteq I$, there is $n \geq 1$ such that $a_0 \cdots a_{n-1}=0$.
 
 The following result about perfect rings will be useful.

\begin{fact}\label{fmany}(\cite[Theorem 1]{saek}) Let $R$ be a left perfect ring. There is a bijection between the strongly indecomposable projective left  $R$-modules  and the strongly indecomposable left $R/J(R)$-modules. \end{fact}

%\begin{fact}\label{fmany}(\cite[Theorem 1]{saek})Let $R$ be a left perfect ring.
%\begin{enumerate}
%\item A projective module is  indecomposable if and only if it is strongly indecomposable.
%\item There is a bijection between the strongly indecomposable projective left  $R$-modules  and the simple left $R/J(R)$-modules.
%\item Every flat left $R$-module is a direct sum of strongly indecomposable projective left $R$-modules.
%\end{enumerate}
%\end{fact}
We show that if a ring is left perfect, then the class of flat modules satisfies Hypothesis \ref{hyp1}.

 \begin{lemma}\label{cor-flat}
 If $R$ is left perfect, then the class of flat $R$-modules satisfies Hypothesis \ref{hyp1} with $\si \leq \text{card}(R) +\aleph_0$.
 \end{lemma}
 \begin{proof} Since $R$ is left perfect, the class of flat modules is the same as the class of projective modules. Condition 2.4.(1) holds by \cite[Theo 1]{saek} and the fact that projective indecomposable modules are strongly indecomposable in perfect rings by \cite[Prop 2]{saek}. $\si \leq \ccr$ since every strongly indecomposable projective module is a direct summand of $R$ by \cite[Cor 1]{saek}. It is clear that Condition 2.4.(3) holds as the class of flat modules is closed under direct sums.
 \end{proof}
 
In this case we can transfer categoricity algebraically. 

\begin{theorem}\label{fmain} Assume $R$ is an associative ring with unity. The following are equivalent.
\begin{enumerate}
\item $R\text{-Flat}$ is $\lambda$-categorical in all $\lambda > \text{card}(R) + \aleph_0$
\item $R\text{-Flat}$ is $\lambda$-categorical in some $\lambda  > \text{card}(R) + \aleph_0$. 
\item $R$ is left perfect and $R/J(R) \cong M_n(D)$ for $D$ a division ring and $n \geq 1$. 
\item $R \cong M_n(k)$ for $k$ a local ring such that its maximal ideal is left $T$-nilpotent and $n \geq 1$. 
\end{enumerate}
\end{theorem}
\begin{proof}
We only need to show (2) implies (3) and (3) implies (1). (1) implies (2) is clear and the equivalence between (3) and (4) is \cite[23.23]{lam2}.

$(2) \Rightarrow (3)$: We show first that $R$ is left perfect. Let $M$ be a non-zero countably generated flat $R$-module, hence $\| M \| \leq \text{card}(R) +\aleph_0 < \lambda$. Then $M^{(\lambda)}$ is free by $\lambda$-categoricity, the fact that free modules are flat and the closure of flat modules under direct sums. Therefore, $M$ is projective, so $R$ is left perfect.

Since $R$ is left perfect, we have that $R/J(R)$ is semisimple. Then as in Theorem \ref{mmod} $R/J(R) \cong  M_{n_1}(D_1) \times \cdots \times M_{n_{m}}(D_{m})$ for some $m \in \mathbb{N}$, $D_1, \cdots, D_m$ division rings and $n_1, \cdots, n_m \in \mathbb{N}$.  By Lemma \ref{cor-flat} and Lemma \ref{basicl} it follows that the class of flat modules has a unique strongly indecomposable module up to isomorphisms. Therefore, $m=1$ by Fact \ref{fmany} and Fact \ref{wa}.

$(3) \Rightarrow (1)$: Since $R$ is left perfect and $R\text{-Flat}$  has a unique strongly indecomposable module by Fact \ref{fmany} and Fact \ref{wa}, the results follows from Lemma \ref{cor-flat} and Lemma \ref{basicl}. \end{proof}

\begin{remark}
The equivalence between (1) and (2) of Theorem \ref{fmain} is an instance of Shelah's Categoricity Conjecture. Since the class of flat left $R$-modules is first-order axiomatizable if and only if $R$ is right coherent \cite[Thm4]{saek}, the result does not follow from Morley's Categoricity Theorem or Shelah's generalization. Moreover, the result does not follow directly from any of the positive results known for AECs.
\end{remark}

The above theorem can be used to provide a natural family of examples of AECs categorical in a tail of cardinals which are not first-order axiomatizable. More precisely, if $R \cong M_n(k)$ for $k$ a local ring such that its maximal ideal is left $T$-nilpotent and $R$ is not right coherent then $R\text{-Flat}$ is $\lambda$-categorical in all $\lambda > \text{card}(R) + \aleph_0$ and it is not first-order axiomatizable. We provide three examples. We thank an anonymous referee  for pointing out the third example.

\begin{example}\label{exflat}\
\begin{enumerate}
\item Let $F$ be a countable field and $V$ be an infinite dimensional $F$-vector space. Let $R =T(F, V)$ be the trivial extension, i.e., $R= F \times V$ and one adds coordinate-wise and $(a,\bar{v})(b,\bar{w})=(ab, a\bar{w} + b\bar{v})$. It is known that $R$ is a commutative perfect local ring that is not coherent, see for example \cite[3.3]{roth}.

\item  (\cite[Example 6]{cny}) Let $R=\mathbb{Z}_2[x_1, x_2, \cdots]$ where $x^3_i =0$,  $x^2_i  = m \neq 0$ for every $i$ and  $x_i x_j =0$ for every $i \neq j$. It is shown in \cite[Example 6]{cny} that $R$ is a commutative semiprimary local ring that is not artinian. Semiprimary rings are perfect and for commutative perfect rings being artinian is equivalent to being coherent by \cite[5.3]{alco}. Therefore, $R$ is a commutative perfect local ring that is not coherent.

\item Let $F$ be a field. Let $R$ be the ring of $\omega \times \omega$ lower triangular matrices over $F$ that are constant in the diagonal and that have only finitely many non-zero entries off the diagonal. It is known that $R$ is a left perfect local ring that is not right perfect \cite[23.22]{lam1}. Since $R$ is left perfect and not right perfect, it follows that $R$ is not right coherent by \cite[14.23]{prest}. Therefore, $R$ is a left perfect local ring that is not right coherent.
\end{enumerate}
\end{example}

%\begin{remark}
%The above rings pr last example is the first algebraic example of a class categorical in a tail of cardinals which is not first-order axiomatizable.
%\end{remark}

In \cite[3.15]{m3} it was shown that $(R\text{-Flat}, \leq_p)$ is superstable\footnote{An abstract elementary class $\K$ is superstable if and only if it has uniqueness of limit models in a tail of cardinals.} if and only if $R$ is left perfect. The next ring provides a natural example of an AEC which is superstable, not categorical in a tail of cardinals and not first-order axiomatizable. This is the first example of this sort.

\begin{example}\label{last}
 Let $F$ be a countable field and $V$ be an infinite dimensional $F$-vector space. Let  $S =T(F, V)$ be the trivial extension. Let $R = S \times S$. Since $R$ is commutative and $S$ is local with a left $T$-nilpotent maximal ideal, it follows that $R$ is perfect by \cite[23.24]{lam1}.  Hence $(R\text{-Flat}, \leq_p)$  is superstable by \cite{m3}.
 
 It is easy to show that $R/J(R) \cong S \times S/J(S \times S) \cong F \times F$. Moreover, $ F \times F \ncong M_n(D)$ for any $n \in \mathbb{N}$ and $D$ a division ring because $ F \times F$ is commutative and not a field. Hence $R$-Flat is not categorical in a tail of cardinals by Theorem \ref{fmain}.(3).

 Finally, $R$ is not artinian because $S$ is not artinian.  Since $R$ is a commutative perfect ring, then $R$ is not coherent by \cite[5.3]{alco}. Hence $R$-Flat is not first-order axiomatizable.
\end{example}

\section{Future work}

In the main results of this paper, i.e., Theorem \ref{lpi}, Theorem \ref{absp-c} and Theorem \ref{fmain}, we showed that the condition of being categorical in a tail of cardinals can be characterized algebraically for several classes of modules. It is natural to ask if this is the case for other classes of modules.

\begin{question}\label{q1} Let $K$ be the class of injective modules, $\Sigma$-injective modules, pure-injective modules, $\Sigma$-pure-injective modules, semisimple modules  or projective modules. Is being categorical in a tail of cardinals equivalent to a natural ring theoretic property for $K$?
\end{question}

It follows from our results that the answer to the question is positive if the ring is left noetherian, left pure-semisimple, semisimple  or left perfect respectively. 

On a different direction, it would be interesting to determine if categoricity in \emph{some} big cardinal can be transferred to categoricity in \emph{all} big cardinals for the following classes of modules.

\begin{question}\label{q2} Let $K$ be the class of injective modules, pure-injective modules or projective modules. Is being categorical in \emph{all} big cardinals equivalent to being categorical in \emph{some} big cardinal for $K$? 
\end{question}

Observe that we dropped the class of $\Sigma$-injective modules, $\Sigma$-pure-injective modules and semisimple modules from the previous question as we know that in those cases the answer is positive by Corollary \ref{cors}. Again, it follows from our results that the answer to the question is positive if the ring is left noetherian, left pure-semisimple  or left perfect respectively.

Finally, in order to transfer categoricity in the case of locally pure-injective modules, absolutely pure modudes and locally injective modules, we used Vasey's result on how to transfer categoricity in AECs (Fact \ref{sebsc}). A natural question to ask is the following: 

\begin{question}
Is it possible to establish Theorem \ref{lpi}, Theorem \ref{absp} or Theorem \ref{linj} using only module theoretic tools?
\end{question}

A positive solution to this last question would shed light on Question \ref{q2} as it would provide new ways to transfer categoricity in an algebraic setting.

\end{document}